\documentclass[12pt, reqno]{amsart}

\usepackage{graphicx,xcolor,cite,mathtools}
\usepackage{amssymb,amsmath,amsthm,mathrsfs,paralist,bm,esint,setspace}
\RequirePackage[colorlinks,linkcolor=black,citecolor=blue,urlcolor=blue]{hyperref}

\numberwithin{equation}{section}

\newtheorem{theorem}{Theorem}[section]
\newtheorem{lemma}[theorem]{Lemma}
\newtheorem{proposition}[theorem]{Proposition}

\theoremstyle{definition}

\newtheorem{remark}[theorem]{Remark}

\newcommand{\R}{\mathbb{R}}
\newcommand{\E}{\mathbb{E}}
\newcommand{\Z}{\mathbb{Z}}
\newcommand{\N}{\mathbb{N}}
\newcommand{\h}{\mathfrak{h}}
\newcommand{\Cov}{\operatorname{Cov}}
\newcommand{\Var}{\operatorname{Var}}
\newcommand{\Prob}{\mathbb{P}}
\newcommand{\dd}{\,\mathrm{d}}
\newcommand{\e}{\,\mathrm{e}}
\newcommand{\Lcal}{\mathcal{L}}
\newcommand{\Acal}{\mathcal{A}}

\newcommand{\one}{\mathbf{1}}

\DeclarePairedDelimiter{\floor}{\lfloor}{\rfloor}

\title[Two-time spatial decorrelation]{Two-time spatial decorrelation for the flat KPZ fixed point}
\author{Le Chen and Fei Pu}
\date{}
 \address[Le Chen]{Department of Mathematics and Statistics, Auburn University, Auburn, AL 36849, USA.\\
\textit{Email: lzc0090@auburn.edu}
}

\address[Fei Pu]{Laboratory of Mathematics and Complex Systems,
School of Mathematical Sciences, Beijing Normal University, 100875, Beijing, China\\
\textit{Email: fei.pu@bnu.edu.cn}
}

\begin{document}

\begin{abstract}
  We establish quantitative two-time spatial decorrelation for the
  Kardar--Parisi--Zhang fixed point with flat initial data. For every $s,t>0$,
  there exist constants $C,c>0$ such that
  \[
    \big|\Cov(\h(t,x),\h(s,0))\big|
    \le C\exp\{-c|x|^3\},\qquad |x|\ge1.
  \]
  Unlike the fixed-time covariance, which is governed directly by the Airy$_1$
  process, the two-time covariance involves the nonlinear variational evolution
  of the entire earlier height profile. Our proof combines cubic-exponential
  mixing of the Airy$_1$ process with a uniform localization estimate for
  intermediate optimizers in the directed landscape. As a consequence, the
  centered spatial averages, normalized by $N^{1/2}$, converge in
  finite-dimensional distributions to a centered Gaussian process whose
  covariance is the space-integrated two-time correlation of the flat KPZ fixed
  point.
\end{abstract}

\maketitle


\medskip
\tableofcontents
\medskip

\section{Introduction}\label{S:problem}

The Kardar--Parisi--Zhang (KPZ) universality class describes large-scale
fluctuations in random growth, interacting-particle systems, directed polymers,
and last-passage percolation. Its formal
continuum representative is the KPZ equation
\begin{equation}\label{E:KPZ}
  \partial_t h(t,x)=\frac12\partial_x^2h(t,x)
  +\frac12[\partial_x h(t,x)]^2+\xi(t,x)
\end{equation}
for $t>0$ and $x\in\R$, where $\xi$ is space-time white noise. Under the
Cole--Hopf transformation $h=\log u$, equation \eqref{E:KPZ} is related to the
stochastic heat equation
\begin{equation}\label{E:SHE}
  \partial_tu(t,x)=\frac12\partial_x^2u(t,x)+u(t,x)\xi(t,x).
\end{equation}
The universal long-time dynamics of one-dimensional KPZ models is encoded by the
KPZ fixed point, denoted by $\{\h(t,x):(t,x)\in\R_+\times\R\}$. It was
constructed by Matetski, Quastel and Remenik \cite{matetski.quastel.ea:21:kpz}
from TASEP and is invariant under the characteristic $1:2:3$ scaling.
Convergence to the fixed point has since been established for several models;
see, for example, \cite{aggarwal.corwin.ea:24:kpz, virag:20:heat, wu:23:kpz}.

In this paper we study the KPZ fixed point started from flat initial data,
$\h(0,x)\equiv0$. Our principal objective is quantitative \emph{two-time spatial
decorrelation}. At a fixed time, the spatial process is a rescaled Airy$_1$
process. Basu, Busani and Ferrari \cite[Theorem~1.1]{basu.busani.ea:23:on}
proved the sharp logarithmic asymptotic
\[
  \Cov(\Acal_1(0),\Acal_1(u))
  =\exp\left\{-\left(\frac43+o(1)\right)u^3\right\},
  \qquad u\to\infty.
\]
After the one-time scaling identification, this gives a fixed-time logarithmic
covariance exponent proportional to $-|x|^3/t^2$. Thus, at the level of spatial
decay, the fixed-time formula and our two-time upper bound below have the same
cubic-exponential structure. For $0<s<t$, however, the two-time quantity
\[
  \Cov(\h(t,x),\h(s,0))
\]
is of a different nature: the later height $\h(t,x)$ depends nonlinearly on the
entire time-$s$ profile through the variational evolution
\[
  \h(t,x)=\sup_{y\in\R}
  \{\h(s,y)+\Lcal(y,s;x,t)\}.
\]
The agreement of the cubic exponent does not make the two-time result a direct
consequence of the fixed-time formula. Indeed, $\h(t,x)$ is not a second distant
point of the same Airy$_1$ profile, but a nonlinear variational functional of
the entire time-$s$ profile, coupled to an independent future
directed-landscape increment. Consequently, the Airy$_1$ covariance theorem
supplies the fixed-time case and a benchmark for the expected exponent; the
genuinely two-time estimate instead requires an Airy$_1$ mixing input together
with optimizer localization.
The central issue is to show that the part of the time-$s$ profile relevant to
$\h(t,x)$ remains localized near $x$ with cubic-exponentially high probability
and is consequently far from the point $0$.

Our main result is the following.

\begin{theorem}[Two-time spatial decorrelation]\label{T:main-decay}
  For every $s,t>0$, there exist constants $C = C(s,t)>0$ and $c = c(s,t)>0$
  such that
  \begin{equation}\label{E:main-decay}
    \big|\Cov(\h(t,x),\h(s,0))\big|
    \le C\exp\{-c|x|^3\},\qquad |x|\ge1.
  \end{equation}
\end{theorem}

The cubic exponent is consistent with the characteristic transversal
localization tails of KPZ geometry. The theorem gives an upper bound at this
natural exponent; we do not identify the sharp constant or prove a matching
lower bound. By the $1:2:3$ scaling of the flat KPZ fixed point,
Theorem~\ref{T:main-decay} may equivalently be written in a scaling-covariant
form. If $0<s\le t$ and $\rho=s/t$, then there are $C_\rho,c_\rho>0$ such that
\begin{equation}\label{E:scaling-covariance-bound}
  \big|\Cov(\h(t,x),\h(s,0))\big|
  \le t^{2/3}C_\rho
  \exp\left\{-c_\rho\frac{|x|^3}{t^2}\right\},
  \qquad |x|\ge t^{2/3}.
\end{equation}
The constants in this formulation depend only on the ratio of the two times.

To place the result in broader context, Widom \cite{widom:04:on} proved the
$x^{-2}$ covariance decay of the Airy$_2$ process, whose polynomial behavior
contrasts with the Airy$_1$ asymptotic recalled above. Spatial decorrelation for
the KPZ equation with narrow-wedge and flat initial data was recently studied in
\cite{gu.pu:25:spatial, chen.jimenez:26:spatial}. Multipoint distribution
formulas for the KPZ fixed point are available through the work of Liu
\cite{liu:22:multipoint}; see also \cite{johansson.rahman:21:multitime}. It is
not apparent how to extract \eqref{E:main-decay} from these formulas. To the
best of our knowledge, Theorem~\ref{T:main-decay} is the first quantitative
two-time spatial covariance estimate for the flat KPZ fixed point.

The proof identifies a simple decorrelation mechanism:
\begin{align*}
  & \text{Airy$_1$ mixing}\ +\ \text{directed-landscape localization}\ \\
  & \qquad\qquad\qquad \qquad  \qquad \qquad \qquad \quad  \Longrightarrow\
         \text{two-time KPZ decorrelation}.
\end{align*}
For $0<s<t$ and $\eta\in(0,1)$, we truncate the variational formula to
$|y-x|\le\eta|x|$ and write
\begin{align*}
  \Cov(\h(t,x),\h(s,0))
  ={} & \Cov(H_{\eta|x|}(t,s,x),\h(s,0)) \\
      & +\Cov(\h(t,x)-H_{\eta|x|}(t,s,x),\h(s,0));
\end{align*}
see \eqref{E:truncated} for the definition of the truncated observable $H$.
Conditionally on the future landscape, the first term is a covariance between
$\h(s,0)$ and a functional of the time-$s$ profile supported at distance at
least $(1-\eta)|x|$ from the origin. It is controlled by a cubic-exponential
half-line $\alpha$-mixing estimate for the Airy$_1$ process. The second term is
supported on the event that a time-$s$ optimizer leaves the localization window.
A uniform directed-landscape localization estimate gives the same
cubic-exponential bound.

As an application, we study the spatial fluctuation field. Define
\[
  S_{N,t} \coloneqq \int_0^N\big(\h(t,x)-\E[\h(t,x)]\big)\dd x.
\]

\begin{theorem}[Multi-time spatial central limit theorem]\label{T:CLT}
As $N\to\infty$,
\[
  \left\{N^{-1/2}S_{N,t}\right\}_{t>0}
  \xrightarrow{\mathrm{f.d.d.}}
  \{\mathcal G(t)\}_{t>0},
\]
where $\mathcal G$ is a centered Gaussian process with covariance
\begin{equation}\label{E:limit-covariance}
  \Cov(\mathcal G(t),\mathcal G(s))
  =\chi(t,s) \coloneqq \int_\R\Cov(\h(t,x),\h(s,0))\dd x.
\end{equation}
\end{theorem}

The kernel $\chi$ is the space-integrated two-time correlation, or dynamic
susceptibility, of the flat KPZ fixed point. It is symmetric, positive
semidefinite, and homogeneous of degree $4/3$:
\begin{equation}\label{E:chi-scaling}
  \chi(\lambda t,\lambda s)=\lambda^{4/3}\chi(t,s),
  \qquad \lambda>0.
\end{equation}
On the diagonal,
\[
  \chi(t,t)=2^{4/3}t^{4/3}\sigma^2,
  \qquad
  \sigma^2 \coloneqq \int_\R
  \Cov(\Acal_1(x),\Acal_1(0))\dd x.
\]
At a single fixed time, the corresponding CLT follows from
\cite[Theorem~1.1]{pu:23:ergodicity}. Theorem~\ref{T:CLT} gives simultaneous convergence at
finitely many times. Its proof combines Theorem~\ref{T:main-decay} with
spatial stationarity, association, and Newman's central limit theorem
\cite{newman:80:normal}.

\begin{remark}[Further questions]\label{rem:further-questions}
  The results above leave several natural problems open. First, it would be
  desirable to determine the sharp asymptotic behavior of
  $\Cov(\h(t,x),\h(s,0))$ as $|x|\to\infty$, including a matching lower bound
  and the leading constant. Second, identifying the dynamic susceptibility
  $\chi(t,s)$ would determine the law of the limiting Gaussian process in
  Theorem~\ref{T:CLT}. The homogeneity \eqref{E:chi-scaling} alone does not
  imply that this process is fractional Brownian motion; one would additionally
  need, for example, stationary increments and the corresponding covariance
  identity. Third, a functional central limit theorem remains open. Its proof
  would require tightness in the time variable, beyond the finite-dimensional
  convergence established here. Finally, it is natural to ask whether the
  mixing-plus-localization mechanism extends to other initial data and to more
  general KPZ fixed-point evolutions.
\end{remark}

The paper is organized as follows. Section~\ref{S:inputs} records the
variational representation and introduces the truncated observable.
Section~\ref{S:airy1-mixing} proves the quantitative Airy$_1$ $\alpha$-mixing
estimate needed for the distant-profile term. Section~\ref{S:localization}
establishes uniform localization of intermediate directed-landscape optimizers.
The covariance theorem is proved in Section~\ref{S:covariance-proof}; the
Gaussian fluctuation theorem and properties of $\chi$ are discussed in
Section~\ref{S:fluctuations}.

\section{KPZ fixed point, directed landscape and truncation}\label{S:inputs}

We collect the variational input used throughout the paper. The KPZ fixed point
admits a variational representation in terms of the directed landscape, which
was introduced by Dauvergne, Ortmann and Vir\'ag
\cite{dauvergne.ortmann.ea:22:directed}. Let 
\begin{align*}
  \Lcal(y,s; x, t), \quad x,y\in \R, s<t
\end{align*}
denote the directed landscape. It is the universal space-time random geometry
arising as the $1:2:3$ scaling limit of models in the KPZ universality class;
see \cite{dauvergne.ortmann.ea:22:directed} for more information.  Nica, Quastel
and Remenik \cite[Corollary~4.2]{nica.quastel.ea:20:one-sided*1} show that the KPZ fixed point
started from an initial profile $\h_0$ satisfies 
\begin{align} \label{E:var2}
  \h(t,x)=\sup_{y\in \R}\left\{\h_0(y) + \Lcal(y,0; x,t) \right\}.
\end{align}
More generally, the directed landscape has independent increments and satisfies
metric composition; see
\cite[Definition 10.1(II)--(III)]{dauvergne.ortmann.ea:22:directed}. Hence, for
$s<t$,
\begin{align}\label{E:variational}
          \h(t,x)=\sup_{y\in \R}\left\{\h(s, y) + \Lcal(y,s; x,t) \right\},
\end{align}
where the future directed landscape increment $\{\Lcal(y,s; x, t)\}$ is
independent of $\h(s,y)$.

The flat KPZ fixed point is jointly stationary under common spatial shifts:
by \cite[Theorem~4.5(iv)]{matetski.quastel.ea:21:kpz},
\begin{equation}\label{E:joint-spatial-stationarity}
  \{\h(t,x+a):t>0,\ x\in\R\}
  \overset{\mathrm d}=\{\h(t,x):t>0,\ x\in\R\},
  \qquad a\in\R.
\end{equation}

The variational formula \eqref{E:variational} is the starting point to analyze
the two-time covariance of the flat KPZ fixed point. For $t>s>0$, $r>0$ and
$x\in \R$, we introduce the truncated observable
\begin{equation}
  H_r(t,s,x)=\sup_{|y-x|\le r}\{\h(s,y)+\Lcal(y,s ;x,t )\}.
  \label{E:truncated}
\end{equation}
In order to study $\Cov(\h(t,x), \h(s,0))$, we write
\begin{align}\label{E:cov12}
\Cov(\h(t,x), \h(s,0)) &= \Cov(\h(t,x)- H_{\eta |x|}(t,s,x), \h(s,0)) \nonumber \\
&\quad \quad + \Cov(H_{\eta|x|}(t,s,x), \h(s,0))
\end{align}
with $\eta\in (0, 1)$. To estimate the second covariance on the right-hand side
of \eqref{E:cov12}, we need the following uniform moments bound on $H$. 

\begin{lemma}\label{L:truncated-moments}
  For every $0<s<t$ and every $\eta\in(0,1)$,
  \begin{equation}
    \sup_{|x|\ge1}
    \|H_{\eta |x|}(t,s,x)\|_4 <\infty.
    \label{E:truncated-moments}
  \end{equation}
\end{lemma}

\begin{proof}
  Denote
  \[
    W_x=\h(s,x)+\Lcal(x,s;x,t).
  \]
  Since $x\in[x-r,x+r]$, the single candidate $W_x$ is included in the
  truncated supremum.  In addition, the truncated supremum is bounded above by
  the full variational supremum and hence we have
  \[
    W_x\le H_{\eta|x|}(t,s,x)\le \h(t,x).
  \]
  Therefore, $|H_{\eta|x|}(t,s,x)|\le |\h(t,x)|+|W_x|$.  By spatial
  stationarity of the flat KPZ fixed point, the laws of $\h(t,x)$ and $\h(s,x)$
  do not depend on $x$.  Similarly, by the spatial stationarity of the directed
  landscape (see \cite[Lemma 10.2 (2)]{dauvergne.ortmann.ea:22:directed}),
  $\Lcal(x,s;x,t)$ has the same law as $\Lcal(0,s;0,t)$.   Hence the $L^4$
  norms of $\h(t,x)$ and $W_x$ do not depend on $x$. Furthermore, the random
  variables $\h(t,0)$ and $\Lcal(0,s;0,t)$ have the scaled GOE and GUE
  Tracy-Widom distributions respectively, and hence have finite moments. This
  gives the claimed uniform bound.
\end{proof}

The role of the truncation is the following. When we estimate the covariance $
\Cov(H_{\eta|x|}(t,s,x), \h(s,0))$, we will first condition on the sigma-field
$\sigma(\Lcal(y,s;x',t):x',y\in\R)$. In that case, $H_{\eta|x|}(t,s,x)$ is
measurable with respect to the sigma-field $\sigma(\h(s,y): |y-x|\leq \eta|x|)$.
Note that $|y-x|\leq \eta|x|$ implies $x-\eta|x|\leq y\leq \eta|x|+x$ and hence
$y$ is away from origin with distance $(1-\eta)|x|$. In other words, we need to
analyze the covariance between $\h(s,0)$ and an $L^4$-integrable random variable
which is measurable with respect to $\sigma(\h(s,y): y\geq (1-\eta)x)$ (suppose
that $x$ is positive). This will be done by the $\alpha$-mixing of the Airy$_1$
process, as we will discuss in the next section. 

\section{\texorpdfstring{Airy$_1$ $\alpha$-mixing}{Airy1 alpha-mixing}}\label{S:airy1-mixing}

The first quantitative input in the proof of Theorem~\ref{T:main-decay} is an
$\alpha$-mixing estimate for the flat KPZ fixed point at a fixed time. This
follows essentially from the $\alpha$-mixing estimate for the Airy$_1$ process.
Note that it is proved in \cite{basu.bhattacharjee:24:limit} that the Airy$_1$
process is strong mixing (see \cite[Remark 1]{basu.bhattacharjee:24:limit}).
Indeed, as we shall see, the estimate on the probability in Remark 1 of
\cite{basu.bhattacharjee:24:limit} leads to a cubic-exponential decay for the
$\alpha$-mixing of the Airy$_1$ process.

Let us introduce some notations of the relevant exponential LPP model. Consider
the last passage percolation on $\mathbb{Z}^2$ with i.i.d. $\rm Exp(1)$ passage
times on the vertices.  For $s\in \R$, let 
\[
  u_N(s)= (N-\floor{s(2N)^{2/3}}, N+\floor{s(2N)^{2/3}})\in \mathbb{Z}^2.
\]
Denote by $L_r$ the line $\{(x,y) \in \mathbb{Z}^2:x+y=r\}$. Let $T_N^*(s)$
denote the last passage time from $L_0$ to $u_N(s)$ (i.e., the maximal weight
among all paths that start at some point in $L_0$ and end at $u_N(s)$ excluding the
last vertex).

\begin{proposition}\label{P:airy1-alpha-mixing}
  For each fixed $s>0$, there are constants $C_s,c_s>0$ such that the process
  $\{\h(s,y): y\in \R\}$ has the $\alpha$-mixing coefficient
  \begin{align}
    \alpha_s(d)& \coloneqq 
    \sup_a
    \sup_{\substack{
    E_-\in\sigma(\h(s,y):y\le a)\\
    E_+\in\sigma(\h(s,y):y\ge a+d)}}
    |\Prob(E_-\cap E_+)-\Prob(E_-)\Prob(E_+)| \nonumber \\
    &
    \le
    C_s \e^{-c_s d^3/s^2},
    \qquad d\ge0.
    \label{E:airy1-alpha-mixing}
  \end{align}
\end{proposition}

\begin{proof}
  By the scaling invariance of the KPZ fixed point (see \cite[Theorem
  4.5(i)]{matetski.quastel.ea:21:kpz}), the process $\{\h(s,y): y\in \R\}$ has
  the same distribution as $\{(2s)^{1/3}\mathcal{A}_1((2s)^{-2/3}y): y\in \R\}$.
  Moreover, by the stationarity of the Airy$_1$ process, in order to prove
  \eqref{E:airy1-alpha-mixing}, it suffices to prove that for some $C, c>0$
  \begin{equation}
    \alpha(d) \coloneqq 
    \sup_{\substack{
    E_-\in\sigma(\mathcal{A}_1(y):y\le 0)\\
    E_+\in\sigma(\mathcal{A}_1(y):y\ge d)}}
    |\Prob(E_-\cap E_+)-\Prob(E_-)\Prob(E_+)|
    \le
    C \e^{-c d^3},
    \quad d\ge0.
    \label{E:airy1-alpha-mixing2}
  \end{equation}

  It is enough to prove \eqref{E:airy1-alpha-mixing2} for $d\ge d_0$, where
  $d_0$ is a fixed constant for which the separation estimate in
  \cite[Remark~1]{basu.bhattacharjee:24:limit} applies. The remaining bounded
  range follows from the trivial bound $\alpha(d)\le1$ after increasing $C$.

  Fix finite coordinate sets $u_1<\cdots<u_m\le0$ and
  $d\le v_1<\cdots<v_n$, and write
  \[
    U_N \coloneqq (X_N(u_1),\ldots,X_N(u_m)),
    \qquad
    V_N \coloneqq (X_N(v_1),\ldots,X_N(v_n)),
  \]
  where
  \[
    X_N(s) \coloneqq \frac{T_N^*(s)-4N}{2^{4/3}N^{1/3}}.
  \]
  Put $q_N=(2N)^{2/3}$, write $\psi(z_1,z_2)=z_1-z_2$, and set
  \begin{align*}
    b_{N,d}&=-\floor{dq_N},\\
    \mathsf H^+_{N,d}&=\{z\in\Z^2:\psi(z)\ge b_{N,d}+1\},\\
    \mathsf H^-_{N,d}&=\{z\in\Z^2:\psi(z)\le b_{N,d}\}.
  \end{align*}
  These half-planes are disjoint and cover $\Z^2$. Moreover,
  \[
    \psi(u_N(s))=-2\floor{s q_N},
  \]
  so, for all sufficiently large $N$, the endpoints with $s\le0$ lie in
  $\mathsf H^+_{N,d}$ and those with $s\ge d$ lie in $\mathsf H^-_{N,d}$. Thus
  the separator has rotated coordinate $\psi\simeq-dq_N$ with our endpoint
  convention. Remark~1 of \cite{basu.bhattacharjee:24:limit} prints the
  separator with the opposite sign. With the displayed definition of $u_N(s)$,
  the sign-compatible, side-specific version of its construction is the one
  used here. The transverse-excursion estimate invoked there is unchanged by
  interchanging the two coordinate axes in the i.i.d. exponential environment,
  and therefore gives the same cubic-exponential bound for this orientation.

  Writing $\omega(z)$ for the weight at $z\in\Z^2$, define, for $s\le0$ and
  $s\ge d$, respectively,
  \begin{align*}
    T_N^{**,+}(s)
      &\coloneqq
        \max_{\substack{\pi:L_0\to u_N(s)\\
                        \pi\subset\mathsf H^+_{N,d}}}
        \sum_{z\in\pi\setminus\{u_N(s)\}}\omega(z),\\
    T_N^{**,-}(s)
      &\coloneqq
        \max_{\substack{\pi:L_0\to u_N(s)\\
                        \pi\subset\mathsf H^-_{N,d}}}
        \sum_{z\in\pi\setminus\{u_N(s)\}}\omega(z).
  \end{align*}
  Let
  \[
    X_N^{**,\pm}(s)
    \coloneqq\frac{T_N^{**,\pm}(s)-4N}{2^{4/3}N^{1/3}}
  \]
  and set
  \begin{align*}
    U_N^{**}&\coloneqq
      (X_N^{**,+}(u_1),\ldots,X_N^{**,+}(u_m)),\\
    V_N^{**}&\coloneqq
      (X_N^{**,-}(v_1),\ldots,X_N^{**,-}(v_n)).
  \end{align*}
  The first vector depends only on the weights in $\mathsf H^+_{N,d}$ and the
  second only on those in $\mathsf H^-_{N,d}$. Hence $U_N^{**}$ and $V_N^{**}$
  are independent.

  Let $\Gamma_N(s)$ be the almost surely unique unrestricted line-to-point
  geodesic from $L_0$ to $u_N(s)$. Planar ordering of these geodesics gives
  \begin{align*}
    \{U_N\ne U_N^{**}\}
      &=\{\exists i:\Gamma_N(u_i)\not\subset\mathsf H^+_{N,d}\}
        =\{\Gamma_N(u_m)\not\subset\mathsf H^+_{N,d}\},\\
    \{V_N\ne V_N^{**}\}
      &=\{\exists j:\Gamma_N(v_j)\not\subset\mathsf H^-_{N,d}\}
        =\{\Gamma_N(v_1)\not\subset\mathsf H^-_{N,d}\}.
  \end{align*}
  Here uniqueness ensures that an unrestricted and a restricted passage time
  agree exactly when the unrestricted geodesic stays in the prescribed
  half-plane. The controlling endpoints are at rotated distance at least
  $\floor{dq_N}$ from the separator. The separation estimate used in
  \cite[Remark~1]{basu.bhattacharjee:24:limit}, applied to these two endpoint
  events, therefore yields constants $C,c>0$, independent of $m,n$ and of the
  sampled locations, such that
  \[
    F_{N,d}\coloneqq\{U_N\ne U_N^{**}\}\cup\{V_N\ne V_N^{**}\},\qquad
    \Prob(F_{N,d})\le C \e^{-c d^3}.
  \]
  This holds for all sufficiently large $N$; the threshold may depend on the
  fixed coordinate sets, which is sufficient for the finite-dimensional limit.

  Now let $f:\R^m\to[-1,1]$ and $g:\R^n\to[-1,1]$ be bounded continuous test
  functions.  Since $U_N^{**}$ and $V_N^{**}$ are independent,
  \[
    \E\bigl[f(U_N^{**})g(V_N^{**})\bigr]
    =
    \E\bigl[f(U_N^{**})\bigr]\E\bigl[g(V_N^{**})\bigr].
  \]
  Moreover, on $F_{N,d}^c$ we have $U_N=U_N^{**}$ and $V_N=V_N^{**}$, so
  \[
    \bigl|\E\bigl[f(U_N)g(V_N)\bigr]-\E\bigl[f(U_N^{**})g(V_N^{**})\bigr]\bigr|
    \le 2\Prob(F_{N,d}),
  \]
  and similarly
  \begin{align*}
    &\bigl|\E\bigl[f(U_N)\bigr]-\E\bigl[f(U_N^{**})\bigr]\bigr|
    \le 2\Prob(F_{N,d}),\\
    \quad
   & \bigl|\E\bigl[g(V_N)\bigr]-\E\bigl[g(V_N^{**})\bigr]\bigr|
    \le 2\Prob(F_{N,d}).
  \end{align*}
  Therefore,
  \[
    \bigl|\Cov\bigl(f(U_N),g(V_N)\bigr)\bigr|
    \le 6\Prob(F_{N,d})
    \le 6C \e^{-c d^3}.
  \]

  Passing to the limit $N\to\infty$ through the finite-dimensional convergence
  from \cite[Theorem 1.4]{basu.bhattacharjee:24:limit}, we first obtain the same
  bound for the limit vectors
  \begin{align*}
    & \widetilde U \coloneqq 2^{1/3}(\Acal_1(2^{-2/3}u_1),\ldots,\Acal_1(2^{-2/3}u_m)), \\
    & \widetilde V \coloneqq 2^{1/3}(\Acal_1(2^{-2/3}v_1),\ldots,\Acal_1(2^{-2/3}v_n)).
  \end{align*}
  The factor $2^{1/3}$ does not change sigma-fields, and replacing $(u_i,v_j,d)$
  by $(2^{2/3}u_i,2^{2/3}v_j,2^{2/3}d)$ transfers the same estimate to the
  Airy$_1$ coordinate vectors
  \[
    U \coloneqq (\Acal_1(u_1),\ldots,\Acal_1(u_m)), \qquad
    V \coloneqq (\Acal_1(v_1),\ldots,\Acal_1(v_n)),
  \]
  after absorbing the factor $4=(2^{2/3})^3$ into the constant $c$.  That is,
  \begin{align}\label{E:covfg}
    \bigl|\Cov\bigl(f(U),g(V)\bigr)\bigr|
    \le 6C \e^{-c d^3}
  \end{align}
  for all bounded continuous $f$ and $g$ with values in $[-1,1]$. Let
  $A\in\mathcal B(\R^m)$ and $B\in\mathcal B(\R^n)$. Since every Borel
  probability measure on a Euclidean space is Radon, there are
  $f_k\in C_b(\R^m;[0,1])$ and $g_k\in C_b(\R^n;[0,1])$ such that as $k\to\infty$
  \[
    \E[|f_k(U)-\one_{\{U\in A\}}|]\longrightarrow0,\qquad
    \E[|g_k(V)-\one_{\{V\in B\}}|]\longrightarrow0.
  \]
  Since these random variables take values in $[0,1]$, we have as $k\to\infty$
  \begin{align*}
    &\big|\Cov(f_k(U),g_k(V))
       -\Cov(\one_{\{U\in A\}},\one_{\{V\in B\}})\big|\\
    &\qquad\le
      2\E[|f_k(U)-\one_{\{U\in A\}}|]
      +2\E[|g_k(V)-\one_{\{V\in B\}}|]\longrightarrow0.
  \end{align*}
  Passing to the limit in \eqref{E:covfg} shows that, for every pair of
  finite-cylinder events
  $E_-\in\sigma(\Acal_1(u_1),\ldots,\Acal_1(u_m))$ and
  $E_+\in\sigma(\Acal_1(v_1),\ldots,\Acal_1(v_n))$,
  \[
    \bigl|\Prob(E_-\cap E_+)-\Prob(E_-)\Prob(E_+)\bigr|
    \le 6C \e^{-c d^3}.
  \]
  Furthermore, for bounded $[0,1]$-valued functions $M$ and $N$ which are
  measurable with respect to the two finite-cylinder sigma-fields, we have
    \begin{align}\label{E:MN}
    |\Cov(M,N)|
    =
    \left|\int_0^1\int_0^1
    \Cov(\one_{\{M>r\}},\one_{\{N>q\}})\, \dd r\dd q\right| \leq 6C \e^{-c d^3}.
  \end{align}
  
  Choose increasing finite rational sets $R_k^-\subset
  \mathbb{Q}\cap(-\infty,0]$ and $R_k^+\subset \mathbb{Q}\cap[d,\infty)$ such
  that
  \[
    \bigcup_{k\ge1} R_k^- = \mathbb{Q}\cap(-\infty,0], \qquad
    \bigcup_{k\ge1} R_k^+ = \mathbb{Q}\cap[d,\infty),
  \]
  and set
  \[
    \mathcal{F}_k^- \coloneqq \sigma(\Acal_1(q):q\in R_k^-), \qquad
    \mathcal{F}_k^+ \coloneqq \sigma(\Acal_1(q):q\in R_k^+).
  \]
  By the continuity of the Airy$_1$ sample paths,
  \[
    \sigma(\Acal_1(y):y\le 0)=\sigma\!\Bigl(\bigcup_{k\ge1}\mathcal{F}_k^-\Bigr), \qquad
    \sigma(\Acal_1(y):y\ge d)=\sigma\!\Bigl(\bigcup_{k\ge1}\mathcal{F}_k^+\Bigr).
  \]
  Now fix events $E_-\in\sigma(\Acal_1(y):y\le0)$ and
  $E_+\in\sigma(\Acal_1(y):y\ge d)$, and let
  \[
    M_k^- \coloneqq \E[\one_{E_-}\mid\mathcal{F}_k^-], \qquad
    M_k^+ \coloneqq \E[\one_{E_+}\mid\mathcal{F}_k^+].
  \]
  Then $0\le M_k^{\pm}\le1$, each $M_k^\pm$ is measurable with respect to a
  finite rational-cylinder sigma-field, and L\'{e}vy's martingale convergence
  theorem (see \cite[Section 14.2]{williams:91:probability}) gives
  $M_k^-\to\one_{E_-}$ and $M_k^+\to\one_{E_+}$ almost surely and in $L^1$.
  Applying \eqref{E:MN} to $M_k^-$ and $M_k^+$ and then passing to the limit
  yields
  \[
    \bigl|\Prob(E_-\cap E_+)-\Prob(E_-)\Prob(E_+)\bigr|
    =
    \lim_{k\to\infty} \bigl|\Cov(M_k^-,M_k^+)\bigr|
    \le 6C \e^{-c d^3}.
  \]
  This proves \eqref{E:airy1-alpha-mixing2} and hence completes the proof of
  Proposition~\ref{P:airy1-alpha-mixing}.
\end{proof}

\begin{lemma}\label{L:alpha-to-l4}
  Fix $s>0$. If $I\subset(-\infty,a]$ and $J\subset[a+d,\infty)$ are Borel sets,
  and if $F\in L^4(\sigma(\h(s,y):y\in I))$ and $G\in L^4(\sigma(\h(s, y):y\in
  J))$, then
  \begin{equation}
    |\Cov(F,G)|
    \le
    C_s \e^{-c_s d^3/s^2}\|F\|_4\|G\|_4.
    \label{E:airy1-mixing}
  \end{equation}
\end{lemma}

\begin{proof}
  According to Davydov's covariance inequality (see
  \cite{davydov:68:convergence}), if $X\in L^p(\mathcal G)$, $Y\in L^q(\mathcal
  H)$, and $p^{-1}+q^{-1}<1$, then
  \[
    |\Cov(X,Y)|
    \le
    12\alpha(\mathcal G,\mathcal H)^{1-p^{-1}-q^{-1}}
    \|X\|_p\|Y\|_q,
  \]
  where 
  \begin{align*}
  \alpha(\mathcal G,\mathcal H)=\sup_{A\in \mathcal G, B\in \mathcal H}|\Prob(A\cap B)-\Prob(A)\Prob(B)|.
  \end{align*}
    Take $p=q=4$ and apply Proposition~\ref{P:airy1-alpha-mixing} to obtain
    \eqref{E:airy1-mixing}.
\end{proof}

Recall the truncated random variable $H_r(t,s,x)$ defined in
\eqref{E:truncated}.

\begin{proposition}\label{P:mixing-window}
  Fix $\eta\in(0,1)$. Then there are constants $C,c>0$, depending on $s,t,\eta$
  such that for $|x|\ge1$,
  \begin{equation}
    |\Cov(H_{\eta|x|}(t,s, x),\h(s,0))|
    \le C\exp\{-c |x|^3\}.
    \label{E:mixing-window-bound}
  \end{equation}
\end{proposition}

\begin{proof} 
  First suppose that $x\geq1$. Since the sigma-fields
  $\sigma(\Lcal(y,s;x',t):x',y\in\R)$ and $\sigma(\h(s,y):y\in\R)$ are
  independent, for each realization $\omega'$ of the future landscape set
  \[
    X(\omega,\omega')
    \coloneqq\sup_{|y-x|\leq\eta|x|}
      \{\h(s,y)(\omega)+\Lcal(y,s;x,t)(\omega')\}.
  \]
  The moment estimates below justify Fubini's theorem, which gives
  \begin{align*}
    &\Cov(H_{\eta|x|}(t,s,x),\h(s,0))\\
    &\quad=\E_{\omega'}\E_{\omega}\!\left[
      X(\omega,\omega')
      \bigl(\h(s,0)(\omega)-\E[\h(s,0)]\bigr)\right].
  \end{align*}
  Here $\E_\omega$ and $\E_{\omega'}$ denote the expectations with respect to
  the random elements $\omega$ and $\omega'$, respectively. For each $\omega'$,
  the random variable $\omega\mapsto X(\omega,\omega')$ is measurable with
  respect to the sigma-field $\sigma(\h(s,y): y\geq
  (1-\eta)x)$.  In addition, the random variable $\omega\mapsto X(\omega,
  \omega')$ has finite moments.  To see this, using the fact that the directed
  landscape is continuous, we have 
  \begin{align*}
  |X(\omega, \omega')| \leq \sup_{|y-x|\leq \eta|x|} |\h(s,y)|(\omega) + \tilde{C} 
  \end{align*}
  for a positive constant $\tilde{C}$ depending on $\omega'$. Thus, we only need
  to show that the maximum of the absolute value of the Airy$_1$ process over a
  compact interval has finite moments. To simplify the notation, let us show
  that $\max_{x\in[0, 1]}|\mathcal{A}_1(x)|$ has finite moments.  Indeed, we
  have
  \begin{align*}
  \max_{x\in[0, 1]}|\mathcal{A}_1(x)| \leq \left|\max_{x\in[0, 1]}\mathcal{A}_1(x)\right|+ \left|\min_{x\in[0, 1]}\mathcal{A}_1(x)\right|.
  \end{align*}
  Using the estimates on the tail probabilities of $\max_{x\in[0,
  1]}\mathcal{A}_1(x)$ and $\min_{x\in[0, 1]}\mathcal{A}_1(x)$ (see
  \cite[Proposition 6.1]{pu:23:ergodicity} and \cite[Lemma
  3.1]{bhattacharjee.pu:25:macroscopic}), we conclude that $\max_{x\in[0,
  1]}|\mathcal{A}_1(x)|$ has finite moments. By stationarity and a finite cover
  by unit intervals, the same conclusion holds on every fixed compact interval.
  The preceding bound then shows that, for each realization of
  $\Lcal(y,s;x,t)$, the random variable $\omega\mapsto X(\omega,\omega')$ has
  finite moments with respect to $\E_\omega$.

  Hence, for each realization of $\Lcal(y,s;x,t)$, we apply
  Lemma~\ref{L:alpha-to-l4} to obtain
  \begin{align}
  &\left|\Cov(H_{\eta|x|}(t,s, x),\h(s,0))\right| \nonumber\\
  &\qquad \leq \E_{\omega'}\left[\left|\E_{\omega}\left[X(\omega, \omega') (\h(s,0)(\omega)-\E[\h(s,0)])\right]\right| \right] \nonumber\\
  &\qquad = \E_{\omega'} \left[\left|\Cov_{\omega} (X(\omega, \omega'), \h(s,0)(\omega))\right|\right]\nonumber\\
  &\qquad \leq C\e^{-c(1-\eta)^3|x|^3} \|\h(s,0)\|_4 \E_{\omega'}\left[
  \left\{\E_{\omega}\left[\left|X(\omega, \omega')\right|^4\right]\right\}^{1/4}
  \right]. \label{E:cov1}
  \end{align}
  By H\"older's inequality,
  \begin{align}
  \E_{\omega'}\left[
  \left\{\E_{\omega}\left[\left|X(\omega,\omega')\right|^4\right]\right\}^{1/4}
  \right] & \leq  \left\{\E_{\omega'}\left[
  \E_{\omega}\left[\left|X(\omega, \omega')\right|^4\right]
  \right] \right\}^{1/4} \nonumber\\
  &= \|H_{\eta|x|}(t,s, x)\|_4. \label{E:cov2}
  \end{align}
           
  Combining \eqref{E:cov1}, \eqref{E:cov2}, and
  Lemma~\ref{L:truncated-moments} proves \eqref{E:mixing-window-bound} for
  $x\ge1$. If $x\le-1$, then $|y-x|\le\eta|x|$ implies
  $y\le-(1-\eta)|x|$. The same argument applies with the left and right
  half-lines interchanged in Lemma~\ref{L:alpha-to-l4}. This proves the result
  for both signs of $x$.
\end{proof}

In view of \eqref{E:cov12} and Proposition \ref{P:mixing-window}, it remains to
estimate 
\[
  \Cov(\h(t,x)- H_{\eta |x|}(t,s,x), \h(s,0)).
\]
By Cauchy-Schwarz inequality, 
\begin{align*}
  & |\Cov(\h(t,x)- H_{\eta |x|}(t,s,x), \h(s,0))| \\
  & \qquad \qquad\leq \left\|\h(t,x)- H_{\eta |x|}(t,s,x)\right\|_2 \|\h(s,0)\|_2.
\end{align*}
Now we need to handle the truncation error $\left\|\h(t,x)- H_{\eta
|x|}(t,s,x)\right\|_2$. Suppose that supremum in \eqref{E:variational} is
attained at a random point $Y_x$. Then we can write
\begin{align}\label{E:optimizer}
&\left\|\h(t,x)- H_{\eta |x|}(t,s,x)\right\|_2 \nonumber \\
&\qquad=\left\|(\h(t,x)- H_{\eta |x|}(t,s,x)) \one_{\{|Y_x-x| >\eta|x|\}}\right\|_2\nonumber\\
&\qquad \leq \left\|(\h(t,x)- H_{\eta |x|}(t,s,x))\right\|_4 \left\|\one_{\{|Y_x-x| >\eta|x|\}}\right\|_4\nonumber\\
&\qquad \leq \left(\|\h(t,0)\|_4+\|H_{\eta|x|}(t,s,x)\|_4\right)
\Prob(|Y_x-x|>\eta|x|)^{1/4}.
\end{align}
Thus, by Lemma \ref{L:truncated-moments}, the remaining task is to estimate the
probability $\Prob(|Y_x-x| >\eta|x|)$, which will be done in the next section. 

\section{Optimizer localization}
\label{S:localization}

The second quantitative input is an optimizer localization bound for the
time-$s$ variational maximizer in \eqref{E:variational}. We begin with a
uniform localization estimate over a compact family of endpoint pairs.

\begin{proposition}\label{L:compact-family-split-localization}
  Fix $0<s<t$ and put $\rho=s/t$ and $m_{s,t}(z,x)=(1-\rho)z+\rho x$.  Fix an
  endpoint $x_0\in\R$, and let $Y^-_{z,x_0}\le Y^+_{z,x_0}$ be the leftmost and
  rightmost maximizers of
  \[
    y\mapsto \Lcal(z,0;y,s)+\Lcal(y,s;x_0,t).
  \]
  For all $R,U\ge1$,
  \begin{align}
    &\Prob\Bigg(
      \begin{aligned}
      &\{\exists |z-x_0|\le R:
        Y^+_{z,x_0}-m_{s,t}(z,x_0)>U\}\\
      &{}\cup\{\exists |z-x_0|\le R:
        Y^-_{z,x_0}-m_{s,t}(z,x_0)<-U\}
      \end{aligned}
    \Bigg)\notag\\
    &\qquad \le
    C\Big(1+\frac{R}{U}\Big)^2\exp\{-cU^3/t^2\},
    \label{E:compact-family-split-localization}
  \end{align}
  where $C,c$ depend only on $s,t$.
\end{proposition}

\begin{proof}
  Fix a deterministic endpoint pair $(z,x_0)$ first.  By \cite[Lemma
  10.8]{dauvergne.ortmann.ea:22:directed}, the profile
  \[
    f_{z,x_0}(y) \coloneqq \Lcal(z,0;y,s)+\Lcal(y,s;x_0,t)
  \]
  attains its maximum, and $Y^-_{z,x_0},Y^+_{z,x_0}$ are the leftmost and
  rightmost maximizers. Using the scaling property of directed landscape (see
  \cite[Lemma 10.2(5)]{dauvergne.ortmann.ea:22:directed}), we write
  \begin{align*}
    f_{z,x_0}(y)=t^{1/3}\left(\Lcal(t^{-2/3}z,0;t^{-2/3}y,\rho)+\Lcal(t^{-2/3}y,\rho; t^{-2/3}x_0,1)\right).
  \end{align*}
  Denote $m=m_{s,t}(z,x_0)$ and introduce the variable $u$ defined by
  \[
    u=t^{-2/3}\bigl(y-m\bigr).
  \]
  It follows that
  \begin{align*}
   & f_{z,x_0}(m+t^{2/3}u)=t^{1/3}\big(\Lcal(t^{-2/3}z,0;t^{-2/3}m+u,\rho) \\
   & \qquad \qquad\qquad\qquad \qquad\qquad +\Lcal(t^{-2/3}m+u,\rho; t^{-2/3}x_0,1)\big).
  \end{align*}
  By the spatial stationarity of $\Lcal$, 
  \begin{align*}
    \Lcal(t^{-2/3}z,0;t^{-2/3}m+u,\rho) \overset{d}{=} \Lcal(0,0;t^{-2/3}(m-z)+u,\rho)
  \end{align*}
  and 
  \begin{align*}
    & \Lcal(t^{-2/3}m+u,\rho; t^{-2/3}x_0,1) \\
    & \qquad\overset{d}{=} \Lcal(-t^{-2/3}(m-x_0)\frac{\rho}{1-\rho}+u,\rho; \frac{-1}{1-\rho}t^{-2/3}(m-x_0),1).
  \end{align*}
  Applying the skew stationarity of $\Lcal$ (choosing $c=-t^{-2/3}(m-z)/\rho$
  in \cite[Lemma 10.2(4)]{dauvergne.ortmann.ea:22:directed}), we see that
  \begin{align}\label{E:law1}
    & \Lcal(0,0;t^{-2/3}(m-z)+u,\rho) \nonumber \\
    & \qquad\overset{d}{=} \Lcal(0,0;u,\rho)+ \rho^{-1}\left((t^{-2/3}(m-z)+u)^2-u^2\right).
  \end{align}
  Similarly, using again skew stationarity of $\Lcal$ (with
  $c=t^{-2/3}(m-x_0)/(1-\rho)$ in \cite[Lemma
  10.2(4)]{dauvergne.ortmann.ea:22:directed}), we deduce that
  \begin{align}\label{E:law2}
    & \Lcal(-t^{-2/3}(m-x_0)\frac{\rho}{1-\rho}+u,\rho; \frac{-1}{1-\rho}t^{-2/3}(m-x_0),1) \nonumber \\
    & \quad\qquad\quad \overset{d}{=} \Lcal(u, \rho; 0,1)+\frac{1}{1-\rho}\left((t^{-2/3}(m-x_0)+u)^2- u^2\right).
  \end{align}
  Since the directed landscape increments are independent on disjoint time
  intervals, we conclude from \eqref{E:law1} and \eqref{E:law2} that
  \begin{align*}
    t^{-1/3}f_{z,x_0}(m+t^{2/3}u)
    & \overset{d}{=} \Lcal(0,0;u,\rho)+\Lcal(u, \rho; 0,1)             \\
    & \qquad + \rho^{-1}\left((t^{-2/3}(m-z)+u)^2-u^2\right)           \\
    & \qquad + \frac{1}{1-\rho}\left((t^{-2/3}(m-x_0)+u)^2- u^2\right) \\
    & =\Lcal(0,0;u,\rho)+\Lcal(u, \rho; 0,1)                           \\
    & \qquad + \rho^{-1}t^{-4/3}(m-z)^2\\
    & \qquad + (1-\rho)^{-1}t^{-4/3}(m-x_0)^2,
  \end{align*}
  where the second equality holds by the definition of $m=m_{s,t}(z,x_0)$.
  The positive factor $t^{1/3}$ and the final additive term, which is independent
  of $u$, do not affect the maximizers.

  The two increments $\Lcal(0,0;u,\rho)$ and $\Lcal(u, \rho; 0,1)$ are
  independent by \cite[Definition~10.1(II)]{dauvergne.ortmann.ea:22:directed}.
  By the Airy sheet marginal property in \cite[Definition
  10.1(I)]{dauvergne.ortmann.ea:22:directed}, the first term,
  $\Lcal(0,0;u,\rho)$, as a function of $u$, is a parabolic Airy process of
  scale $\rho^{1/3}$; by flip symmetry
  \cite[Lemma~10.2(3)]{dauvergne.ortmann.ea:22:directed}, the second term has
  the law of an independent parabolic Airy process of scale $(1-\rho)^{1/3}$.
  This is exactly the two-Airy-process object in
  \cite[Lemma~9.5]{dauvergne.ortmann.ea:22:directed}.  If $\sigma_\rho =
  \min\left\{\rho^{1/3}, (1-\rho)^{1/3}\right\}$ and $S^\pm_\rho$ denotes the
  leftmost/rightmost maximizer of this normalized sum, then
  $Y^\pm_{z,x_0}-m_{s,t}(z,x_0)$ has the same law as $t^{2/3}S^\pm_\rho$.

  Lemma~9.5 of \cite{dauvergne.ortmann.ea:22:directed} therefore gives, for
  every $M>0$,
  \[
    \Prob\bigl(|S^\pm_{\rho}|>M\bigr)
    \le C\exp\{-cM^3/\sigma_\rho^6\}
  \]
  with universal $C,c$ in that normalized statement. Since $\rho \in (0,1)$ is
  fixed, converting back to the original split point yields the fixed
  endpoint-pair estimate
  \begin{equation}
    \Prob\bigl(|Y^\pm_{z,x_0}-m_{s,t}(z,x_0)|>U\bigr)
    \le C\exp\{-cU^3/t^2\},
    \label{E:fixed-pair-split-tail}
  \end{equation}
  where $C,c$ depend only on the fixed $s,t$.

  It remains to make \eqref{E:fixed-pair-split-tail} uniform over $|z-x_0| \le
  R$. We embed the endpoint segment in a rectangle and apply the
  compact-box-to-grid reduction used in the proof of
  \cite[Lemma~12.2]{dauvergne.ortmann.ea:22:directed}, with the
  northeast/southwest comparisons justified by the monotonicity of the
  leftmost/rightmost maximizers in
  \cite[Proposition~9.2]{dauvergne.ortmann.ea:22:directed}.  Put
  \[
    K_R \coloneqq [x_0-R,x_0+R]\times[x_0-R,x_0+R],
  \]
  so the relevant segment $\{(z,x_0):|z-x_0|\le R\}$ lies inside $K_R$.  Let the
  grid mesh be
  \[
    h \coloneqq \frac{U}{2},
  \]
  and choose the finite grid
  \[
    \Gamma \coloneqq h\mathbb{Z}^2\cap
    \{(x,y): d((x,y),K_R)\le \sqrt{2}h\}.
  \]
  Then $\#\Gamma\le C(1+R/U)^2$.  For $(x,y)\in K_R$, let
  $(x^{\nearrow},y^{\nearrow})$ be the northeast grid point of the cell
  containing $(x,y)$ and let $(x^{\swarrow},y^{\swarrow})$ be the southwest one.
  Because the rightmost and leftmost maximizers are nondecreasing in each
  endpoint coordinate by the metric-composition monotonicity in
  \cite[Proposition~9.2]{dauvergne.ortmann.ea:22:directed}, we have
  \[
    Y^+_{x,y}\le Y^+_{x^{\nearrow},y^{\nearrow}},
    \qquad
    Y^-_{x^{\swarrow},y^{\swarrow}}\le Y^-_{x,y}.
  \]
  Also, since $m_{s,t}(x,y)=(1-\rho)x+\rho y$ is affine with coefficients adding
  to $1$, moving from $(x,y)$ to either adjacent comparison point changes the
  center by at most one mesh step:
  \[
    |m_{s,t}(x^{\nearrow},y^{\nearrow})-m_{s,t}(x,y)|\le h=U/2,
  \]
  and similarly for the southwest point.  Hence an upper-tail failure
  $Y^+_{x,y}-m_{s,t}(x,y)>U$ forces
  \[
    Y^+_{x^{\nearrow},y^{\nearrow}}-m_{s,t}(x^{\nearrow},y^{\nearrow})>U/2,
  \]
  and a lower-tail failure $Y^-_{x,y}-m_{s,t}(x,y)<-U$ forces
  \[
    Y^-_{x^{\swarrow},y^{\swarrow}}-m_{s,t}(x^{\swarrow},y^{\swarrow})<-U/2.
  \]
  Applying the fixed-pair bound \eqref{E:fixed-pair-split-tail} with threshold
  $U/2$ at each grid point and summing over $\Gamma$ gives
  \begin{align*}
    &\Prob\Big(
    \exists (x,y)\in K_R:
    Y^+_{x,y}-m_{s,t}(x,y)>U
    \text{ or }
    Y^-_{x,y}-m_{s,t}(x,y)<-U
    \Big)\\
    &\qquad\qquad\qquad \qquad 
    \le C\Bigl(1+\frac{R}{U}\Bigr)^2 \e^{-cU^3/t^2}.
  \end{align*}
  Restricting from $K_R$ back to the segment $\{(z,x_0):|z-x_0|\le R\}$ proves
  \eqref{E:compact-family-split-localization}. This proves Proposition
  \ref{L:compact-family-split-localization}.
\end{proof}
  
In order to show that the supremum in the variational formula
\eqref{E:variational} is attained at a measurable random point, we need the
following lemma. 

\begin{lemma}\label{L:borel-leftmost-argmax}
  Let
  $
    \mathcal{C}_{\mathrm{coer}}
     \coloneqq \{f\in C(\R): f(y)\to-\infty \text{ as } |y|\to\infty\}.
  $
  For $f\in\mathcal{C}_{\mathrm{coer}}$, let
  \[
    \ell(f) \coloneqq \inf\{y\in\R: f(y)=\sup_{z\in\R} f(z)\}.
  \]
  Then $\ell(f)$ is well defined, and the map $\ell:\mathcal{C}_{\mathrm{coer}}\to\R$
  is Borel for the topology of locally uniform convergence.
\end{lemma}

\begin{proof}
  Since $f$ is continuous and coercive, its argmax set is nonempty and compact,
  so $\ell(f)$ is well defined. For any $a \in \R$,
  \[
    \{\ell(f)<a\}
    =
    \bigcup_{\substack{b\in\mathbb{Q}\\ b<a}}
    \Bigl\{\sup_{q\in\mathbb{Q},\ q\le b} f(q)
    =\sup_{q\in\mathbb{Q}} f(q)\Bigr\}.
  \]
  Indeed, if $\ell(f)<a$, then choose rational $b$ with $\ell(f)<b<a$; the
  maximum is attained on $(-\infty,b]$, so the corresponding suprema agree.
  Conversely, if the suprema agree for some rational $b<a$, continuity and
  coercivity force a maximizer in $(-\infty,b]$, and hence $\ell(f)\le b<a$.
  Both sides are determined by countable suprema of evaluation maps, so $\ell$
  is Borel.  This proves Lemma~\ref{L:borel-leftmost-argmax}.
\end{proof}

\begin{proposition}\label{P:flat-to-point-geodesic-localization}
  Fix $0<s<t$ and $\eta \in (0,1)$. One can choose a measurable time-$s$
  optimizer $Y_x$ in \eqref{E:variational} such that, for $|x| \ge 1$,
  \begin{equation}
    \Prob(|Y_x-x|>\eta |x|)
    \le C\exp\{-c |x|^3\}.
    \label{E:proved-geodesic-localization}
  \end{equation}
\end{proposition}

\begin{proof}
  We start from the flat directed landscape representation in \eqref{E:var2}
  with $\h_0\equiv0$
  \[
    \h(t,x)=\sup_{z\in \R}\Lcal(z,0;x,t).
  \]
  Let $Z_x$ be the leftmost maximizer of $z\mapsto\Lcal(z,0;x,t)$, and write
  \[
    G_x(z) \coloneqq \Lcal(z,0;x,t).
  \]
  By the scaling and stationarity properties of the directed landscape
  \cite[Lemma~10.2]{dauvergne.ortmann.ea:22:directed},
  \[
    G_x(x+t^{2/3}q)\stackrel{d}=t^{1/3}\mathcal S(q,0),
  \]
  where $\mathcal S$ is the Airy sheet.  Lemma~5.3 of
  \cite{dauvergne.sarkar.ea:22:three-halves} yields there exist deterministic
  $c_0, d>0$ such that, uniformly in $q\in\R$,
  \begin{equation}
    \bigl|\mathcal S(q,0)+q^2\bigr|
    \le \mathfrak C+c_0\log^{2/3}(2+|q|),
    \label{E:airy-sheet-envelope-q0}
  \end{equation}
  where $\mathfrak C$ satisfies the tail bound
  \begin{align}\label{E:C}
  \Prob(\mathfrak C>m)\leq c_0\e^{-d m^{3/2}}.
  \end{align}
 In particular, $\mathcal S(q,0)\to-\infty$ as $|q|\to\infty$, so $G_x$ is
 almost surely continuous and coercive. Therefore, its argmax set is nonempty
 and compact.

  If $Z_x=x+t^{2/3}q_x$, then maximality of $Z_x$ against the candidate $q=0$
  gives $\mathcal S(q_x,0)\ge \mathcal S(0,0)$.  Using
  \eqref{E:airy-sheet-envelope-q0} at $q=q_x$ and at $q=0$ yields
  \[
    q_x^2
    \le 2\mathfrak C + 2c_0\log^{2/3}(2+|q_x|)+c_1
  \]
  for a deterministic constant $c_1$.  Choose $Q_0$ so large that
  $2c_0\log^{2/3}(2+u)\le u^2/2$ for $u\ge Q_0$.  Then on $\{|q_x|\ge Q_0\}$,
  \[
    \mathfrak C\ge c_2 q_x^2
  \]
  for a deterministic $c_2>0$.   Hence, we see from \eqref{E:C} that 
  \[
    \Prob(|q_x|>M)
    \le C\exp\{-c M^3\},
    \qquad M>0.
  \]
  Taking $M=\theta |x|/t^{2/3}$ proves that, for every fixed $\theta>0$,
  \begin{align}\label{E:Zx}
    \Prob(|Z_x-x|>\theta |x|)
    \le C\exp\{-c |x|^3\},
  \end{align}
  where the constant $c$ depends on $t$ and $\theta$.

  The preceding coercivity and joint continuity show that
  $G_x\in\mathcal C_{\mathrm{coer}}$. Hence
  Lemma~\ref{L:borel-leftmost-argmax} makes $Z_x=\ell(G_x)$ measurable.
  For $z\in\R$, set
  \[
    F_{z,x}(y)=\Lcal(z,0;y,s)+\Lcal(y,s;x,t).
  \]
  Lemma~10.8 and its proof in
  \cite{dauvergne.ortmann.ea:22:directed} give a single probability-one event on
  which, simultaneously for every $z\in\R$, the profile $F_{z,x}$ is continuous
  and coercive, attains its maximum, and has a nonempty compact argmax set. Thus
  the same conclusions hold after substituting the random endpoint $z=Z_x$.
  Joint continuity of $\Lcal$ and measurability of $Z_x$ make $F_{Z_x,x}$ a
  measurable $\mathcal C_{\mathrm{coer}}$-valued random profile. Define
  $Y_x=\ell(F_{Z_x,x})$. Lemma~\ref{L:borel-leftmost-argmax} then gives the
  measurability of $Y_x$.

  Choose $\theta>0$ so small that $(1-\rho)\theta<\eta/2$, with $\rho=s/t$. By a
  union bound, we write
  \begin{align}\label{E:P1}
  & \Prob(|Y_x-x|>\eta |x|) \nonumber                                                                     \\
  & \qquad\leq     \Prob(|Z_x-x|>\theta |x|) +   \Prob(|Y_x-x|>\eta |x|; |Z_x-x|\leq \theta |x|)\nonumber \\
  & \qquad \leq C\exp\{-c |x|^3\}+   \Prob(|Y_x-x|>\eta |x|; |Z_x-x|\leq \theta |x|)
  \end{align}
  thanks to \eqref{E:Zx}. Notice that the events $|Y_x-x|>\eta|x|$ and
  $|Z_x-x|\le\theta|x|$ imply that
  \begin{align*}
    |Y_x-m_{s,t}(Z_x,x)| \geq & |Y_x-x| -|x -m_{s,t}(Z_x,x)| \\
                         >    & \eta|x|-\eta|x|/2=\eta|x|/2,
  \end{align*}
  where the second inequality holds since
  \[
    |m_{s,t}(Z_x,x)-x|=(1-\rho)|Z_x-x|\le (1-\rho)\theta |x|< \eta |x|/2.
  \]
  Hence, we have
  \begin{align*}
    & \Prob(|Y_x-x|>\eta |x|; |Z_x-x|\leq \theta |x|) \\
    & \qquad \qquad \qquad \leq \Prob(  |Y_x-m_{s,t}(Z_x,x)|>\eta |x|/2; |Z_x-x|\leq \theta |x|).
  \end{align*}
   Let $X_0=X_0(s,t,\eta,\theta)$ be any threshold such that $R \coloneqq \theta
   |x|\ge1$ and $U \coloneqq \eta |x|/2\ge1$ when $|x|\ge X_0$. We claim that
  \begin{align*}
    &\left\{|Y_x-m_{s,t}(Z_x,x)|>U;\ |Z_x-x|\le R\right\}\\
    &\quad\subset
      \Big\{\exists |z-x|\le R:\
      Y_{z,x}^+-m_{s,t}(z,x)>U\\
    &\hspace{2.05in}\text{or }Y_{z,x}^--m_{s,t}(z,x)<-U\Big\},
  \end{align*}
  where $Y_{z,x}^+$ and $Y_{z,x}^-$ are the rightmost and leftmost maximizers in
  Proposition \ref{L:compact-family-split-localization}. To see this, we
  consider $A \coloneqq \{ |Y_x-m_{s,t}(Z_x,x)|>U; |Z_x-x|\leq R\}$ and take
  $\omega\in A$. Set $z=Z_x(\omega)$ and by the construction of $Y_x$ we have
  $Y_x(\omega)=Y_{z,x}^-(\omega)$. Then $\omega\in A$ implies that
  \begin{align*}
  |z-x|\leq R\quad \text{and} \quad |Y_{z,x}^-(\omega) -m_{s,t}(z,x)| >U.
  \end{align*}
  If $ Y_{z,x}^-(\omega) -m_{s,t}(z,x)<-U$, then since $|z-x|\leq R$, we have
  \begin{align*}
    \omega\in \left\{\exists |z'-x| \leq R: Y^-_{z',x}-m_{s,t}(z',x) <-U\right\}.
  \end{align*}
  If $ Y_{z,x}^-(\omega) -m_{s,t}(z,x)>U$, since $Y_{z,x}^+(\omega) \geq
  Y_{z,x}^-(\omega)$, it follows that $ Y_{z,x}^+(\omega) -m_{s,t}(z,x)>U$ and
  hence
     \begin{align*}
  \omega\in \{\exists |z'-x| \leq R: Y^+_{z',x}-m_{s,t}(z',x) >U\}.
  \end{align*}
  Thus, the claim is verified. Therefore,
  Proposition~\ref{L:compact-family-split-localization} gives, for all $|x| \geq
  X_0$ such that $\theta|x|, \eta|x|/2\ge1$,
  \begin{align}\label{E:P2}
    \Prob\left(|Y_x-m_{s,t}(Z_x,x)|>\eta |x|/2,\ |Z_x-x|\le\theta |x|\right)
    \le C e^{-c |x|^3}.
  \end{align}
  The preceding two displays \eqref{E:P1} and \eqref{E:P2} imply
  \eqref{E:proved-geodesic-localization}, after changing constants.  The
  remaining bounded range $1\le |x|\le X_0$ is absorbed by enlarging the
  prefactor $C$.

  Finally, $Y_x$ is a time-$s$ optimizer in the fixed-point variational formula
  \eqref{E:variational}. Indeed, metric composition gives
  \[
    \h(t,x)=\Lcal(Z_x,0;x,t)
    =\Lcal(Z_x,0;Y_x,s)+\Lcal(Y_x,s;x,t),
  \]
  and $\h(s,Y_x)=\sup_w\Lcal(w,0;Y_x,s)\ge\Lcal(Z_x,0;Y_x,s)$.  Comparing with
  the variational identity \eqref{E:variational} forces equality, so $Y_x$
  attains the supremum in \eqref{E:variational}. This proves
  Proposition~\ref{P:flat-to-point-geodesic-localization}.
\end{proof}

\section{Two-time spatial covariance decay}
\label{S:covariance-proof}

We now combine the mixing and localization estimates to prove the main
decorrelation theorem.

\begin{proof}[Proof of Theorem~\ref{T:main-decay}]
  We first treat the case $0<s<t$. Fix $\eta=\frac12$. We proceed as
  in~\eqref{E:cov12}:
  \begin{align}\label{E:covestimate}
    & \left|\Cov(\h(t,x), \h(s,0))\right| \nonumber                                                                                            \\
    & \quad\leq  \left|\Cov(\h(t,x)- H_{\eta |x|}(t,s,x), \h(s,0))\right| + \left|\Cov(H_{\eta|x|}(t,s,x), \h(s,0))\right| \nonumber \\
    & \quad \leq \|\h(t,x)- H_{\eta |x|}(t,s,x)\|_2\| \h(s,0)\|_2+ C\e^{-c |x|^3},
  \end{align}
  where the second inequality holds by the Cauchy-Schwarz inequality and
  Proposition \ref{P:mixing-window}. Let $Y_x$ be a measurable time-$s$
  optimizer in \eqref{E:variational} chosen in Proposition
  \ref{P:flat-to-point-geodesic-localization}. Then we argue as in
  \eqref{E:optimizer} to see that
  \begin{align*}
    & \left\|\h(t,x)- H_{\eta |x|}(t,s,x)\right\|_2                                                      \\
    & \qquad\qquad \leq \left(\|\h(t,0)\|_4+\|H_{\eta|x|}(t,s,x)\|_4\right) \Prob(|Y_x-x|>\eta|x|)^{1/4} \\
    & \qquad\qquad \leq C\e^{-c |x|^3},
  \end{align*}
  where we use Lemma \ref{L:truncated-moments} and Proposition
  \ref{P:flat-to-point-geodesic-localization} in the second inequality. This
  together with \eqref{E:covestimate} proves the assertion when $0<s<t$.

  If $s=t$, the claim follows from the fixed-time identification
  $\{\h(t,x):x\in\R\}\overset{d}{=} \{(2t)^{1/3}\Acal_1((2t)^{-2/3}x):x\in\R\}$
  in \cite[(4.15)]{matetski.quastel.ea:21:kpz}
  and the cubic-exponential covariance estimate for the Airy$_1$ process in
  \cite{basu.busani.ea:23:on}. Finally, if $s>t$, spatial stationarity and
  symmetry of covariance give
  \[
    \Cov(\h(t,x),\h(s,0))=\Cov(\h(s,-x),\h(t,0)),
  \]
  so the already proved case applies with the two times interchanged. This proves
  Theorem~\ref{T:main-decay}.
\end{proof}

\section{Spatial averages and Gaussian fluctuations}\label{S:fluctuations}

Theorem~\ref{T:main-decay} implies that the dynamic susceptibility $\chi$ in
\eqref{E:limit-covariance} is well defined. Joint spatial stationarity
\eqref{E:joint-spatial-stationarity} gives, for every $s,t>0$,
\begin{align}
  \frac1N\Cov(S_{N,t},S_{N,s})
  &=\int_{-N}^{N}\left(1-\frac{|x|}{N}\right)
    \Cov(\h(t,x),\h(s,0))\,\dd x\notag\\
  &\longrightarrow \chi(t,s).
  \label{E:Fejer-covariance}
\end{align}
The convergence follows from Theorem~\ref{T:main-decay} and dominated
convergence. Consequently, for $t_1,\ldots,t_k>0$ and $a\in\R^k$,
\[
  \sum_{i,j=1}^k a_i a_j\chi(t_i,t_j)
  =\lim_{N\to\infty}\frac1N
    \Var\!\left(\sum_{i=1}^k a_iS_{N,t_i}\right)\ge0,
\]
so $\chi$ is positive semidefinite. Spatial stationarity and covariance symmetry
also give $\chi(t,s)=\chi(s,t)$. Finally, the scaling relation
\eqref{E:chi-scaling} follows from the $1:2:3$ scaling of the flat KPZ fixed
point and the change of variables $x=\lambda^{2/3}y$.

We next prove the multi-time central limit theorem.

\begin{proof}[Proof of Theorem \ref{T:CLT}]
  The proof is similar to that of \cite[Theorem 1.1]{mueller.pu:25:spatial}. We
  include the details here for the reader's convenience. Denote
  \begin{align*}
    S_{N,t}=\int_0^N (\h(t,x)-\E[\h(t,0)]) \dd x.
  \end{align*}
  Fix $0<t_1<\ldots<t_k$. By the Cram\'{e}r--Wold theorem, we need to prove that for
  all $a_1, \ldots, a_k\in \R$, as $N\to\infty$
  \begin{align}\label{E:fdd}
    \frac{1}{\sqrt{N}}\sum_{\ell=1}^ka_\ell S_{N,t_\ell}\xrightarrow{\rm d} a_1\mathcal{G}(t_1)+
    \ldots + a_k\mathcal{G}(t_k).
  \end{align}
  Without loss of generality, we assume $N$ is an integer.

  We first show that \eqref{E:fdd} holds for all nonnegative $a_1,\ldots,a_k$.
  Fix $a_1$, $\ldots, a_k\in \R_+$. We write
  \begin{align*}
    \sum_{\ell=1}^k a_\ell S_{N, t_\ell}= \sum_{j=1}^NX_j,
  \end{align*}
  where
  \begin{align*}
    X_j=\int_{j-1}^j\sum_{\ell=1}^k
      a_\ell(\h(t_\ell,x)-\E[\h(t_\ell,0)])\,\dd x,\qquad j\in\Z.
  \end{align*}

    Joint spatial stationarity \eqref{E:joint-spatial-stationarity} implies that
  $(X_j)_{j\in\Z}$ is stationary. To prove association, let $u$ solve the
  stochastic heat equation \eqref{E:SHE} with $u(0)\equiv1$, and for $T>0$ set
  \begin{align*}
    H_T(t,x)=\frac{\log u(Tt,2^{1/3}T^{2/3}x)+Tt/24}
      {2^{-1/3}T^{1/3}},\qquad (t,x)\in(0,\infty)\times\R.
  \end{align*}
  The stochastic heat equation field $u$ is associated by
  \cite[Theorem~A.4]{chen.khoshnevisan.ea:23:central}; since the logarithm and
  the positive affine rescaling above are coordinatewise nondecreasing, $H_T$
  is associated as well. By \cite[Theorem~1.8]{wu:23:kpz}, $H_T$ converges in
  distribution to $\h$ in $C((0,\infty)\times\R)$. Association is preserved
  under weak convergence by (P$_5$) of
  \cite{esary.proschan.ea:67:association}, so the full field $\h$ is associated.
  For any finite collection of indices, approximate the corresponding $X_j$'s
  by common-mesh Riemann sums. Because the coefficients $a_\ell$ are
  nonnegative, these sums are coordinatewise nondecreasing functions of
  finitely many field values. Passing once more to the weak limit shows that
  $(X_j)_{j\in\Z}$ is associated.

  The variables $X_j$ are centered and square integrable. Association gives
  $\Cov(X_0,X_j)\ge0$, while Theorem~\ref{T:main-decay}, Fubini's theorem, and
  joint spatial stationarity give
  \begin{align*}
    \sum_{j\in\Z}\Cov(X_0,X_j)
    &=\sum_{\ell,m=1}^k a_\ell a_m
      \int_\R\Cov(\h(t_\ell,x),\h(t_m,0))\,\dd x\\
    &=\sum_{\ell,m=1}^k a_\ell a_m\chi(t_\ell,t_m)
    <\infty.
  \end{align*}
  Thus the hypotheses of \cite[Theorem~2]{newman:80:normal} hold, and as
  $N\to\infty$,
  \begin{align*}
    \frac{X_1+\ldots +X_N}{\sqrt{N}} \xrightarrow{\rm d} {\rm N}(0, \tau^2)
  \end{align*}
  with
  \[
    \tau^2=\sum_{j\in\Z}\Cov(X_0,X_j)
    =\sum_{\ell,m=1}^k a_\ell a_m\chi(t_\ell,t_m).
  \]
  This proves that
  \eqref{E:fdd} holds for all $a_1, \ldots, a_k\in \R_+$.

  Taking $s=t$ in \eqref{E:Fejer-covariance} shows that
  $\E|S_{N,t}|^2/N\to\chi(t,t)$. Hence
  \[
    \bigl\{N^{-1/2}(S_{N,t_1},\ldots,S_{N,t_k}):N\in\N\bigr\}
  \]
  is tight. Given any sequence $N_j\to\infty$, choose a subsequence $N_j'$ such
  that
  \[
    (N_j')^{-1/2}(S_{N_j',t_1},\ldots,S_{N_j',t_k})
    \xrightarrow{\rm d}(G_1,\ldots,G_k).
  \]
  The convergence already proved for nonnegative coefficients shows that, for
  every $a\in\R_+^k$, the random variable $a\cdot G$ is centered Gaussian with
  variance
  \[
    \sum_{i,j=1}^k a_i a_j\chi(t_i,t_j).
  \]
  Lemma~A.1 of \cite{mueller.pu:25:spatial} therefore implies that
  $(G_1,\ldots,G_k)$ is a centered Gaussian vector. Taking $a=e_i$ identifies
  $\Var(G_i)=\chi(t_i,t_i)$, while taking $a=e_i+e_j$ gives
  \[
    2\Cov(G_i,G_j)
    =\Var(G_i+G_j)-\Var(G_i)-\Var(G_j)
    =2\chi(t_i,t_j).
  \]
  Thus every subsequential limit has the law of
  $(\mathcal G(t_1),\ldots,\mathcal G(t_k))$. Therefore, as $N\to\infty$,
  \begin{align*}
    N^{-1/2}(S_{N,t_1},\ldots,S_{N,t_k})
    \xrightarrow{\rm d}(\mathcal G(t_1),\ldots,\mathcal G(t_k)).
  \end{align*}
  The proof of Theorem \ref{T:CLT} is complete.
\end{proof}

\section*{Acknowledgments} 
L.C. was supported in part by NSF grants DMS-2246850/2443823 and a Collaboration
Grant for Mathematicians (\#959981) from the Simons Foundation. F.P. was
supported in part by National Natural Science Foundation of China (No. 12571153)
and National Key R\&D Program of China (No. 2022YFA 1006500). During the preparation of this work, the authors used large language model–based tools as auxiliary aids in the research and writing process. All mathematical content was reviewed and verified by the authors, who take full responsibility for the manuscript.


\begin{thebibliography}{999}

\bibitem{aggarwal.corwin.ea:24:kpz}
Aggarwal, A., Corwin, I. and Hegde, M. (2024).
KPZ fixed point convergence of the ASEP and stochastic six-vertex models.
Preprint \href{https://arxiv.org/abs/2412.18117}{arXiv:2412.18117}.

\bibitem{basu.bhattacharjee:24:limit}
Basu, R. and Bhattacharjee, S. (2024).
Limit theorems for extrema of Airy processes.
Preprint \href{https://arxiv.org/abs/2406.11826}{arXiv:2406.11826}.

\bibitem{basu.busani.ea:23:on}
Basu, R., Busani, O. and Ferrari, P. L. (2023).
On the exponent governing the correlation decay of the Airy$_1$ process.
\textit{Comm. Math. Phys.} \textbf{398}, no. 3, 1171--1211.
\href{https://mathscinet.ams.org/mathscinet-getitem?mr=4561801}{MR4561801};
\href{https://doi.org/10.1007/s00220-022-04544-1}{doi:10.1007/s00220-022-04544-1}.

\bibitem{bhattacharjee.pu:25:macroscopic}
Bhattacharjee, S. and Pu, F. (2026).
Macroscopic Hausdorff dimension of the level sets of the Airy processes.
\textit{Electron. J. Probab.} \textbf{31}, Paper No. 36, 21 pp.
Preprint \href{https://arxiv.org/abs/2501.00772}{arXiv:2501.00772}.

\bibitem{chen.jimenez:26:spatial}
Chen, L. and Jim\'enez, J. J. (2026).
Spatial covariance of KPZ from flat initial profile.
Preprint \href{https://arxiv.org/abs/2603.14174}{arXiv:2603.14174}.

\bibitem{chen.khoshnevisan.ea:23:central}
Chen, L., Khoshnevisan, D., Nualart, D. and Pu, F. (2023).
Central limit theorems for spatial averages of the stochastic heat equation via Malliavin--Stein's method.
\textit{Stoch. Partial Differ. Equ. Anal. Comput.} \textbf{11}, no. 1, 122--176.
\href{https://mathscinet.ams.org/mathscinet-getitem?mr=4563698}{MR4563698};
\href{https://doi.org/10.1007/s40072-021-00224-8}{doi:10.1007/s40072-021-00224-8}.


\bibitem{dauvergne.ortmann.ea:22:directed}
Dauvergne, D., Ortmann, J. and Vir\'ag, B. (2022).
The directed landscape.
\textit{Acta Math.} \textbf{229}, no. 2, 201--285.
\href{https://mathscinet.ams.org/mathscinet-getitem?mr=4554223}{MR4554223};
\href{https://doi.org/10.4310/acta.2022.v229.n2.a1}{doi:10.4310/acta.2022.v229.n2.a1}.

\bibitem{dauvergne.sarkar.ea:22:three-halves}
Dauvergne, D., Sarkar, S. and Vir\'ag, B. (2022).
Three-halves variation of geodesics in the directed landscape.
\textit{Ann. Probab.} \textbf{50}, no. 5, 1947--1985.
\href{https://mathscinet.ams.org/mathscinet-getitem?mr=4474505}{MR4474505};
\href{https://doi.org/10.1214/22-aop1574}{doi:10.1214/22-aop1574};
\href{https://arxiv.org/abs/2010.12994}{arXiv:2010.12994}.

\bibitem{davydov:68:convergence}
Davydov, J. A. (1968).
The convergence of distributions which are generated by stationary random processes.
\textit{Teor. Verojatnost. i Primenen.} \textbf{13}, 730--737.
\href{https://mathscinet.ams.org/mathscinet-getitem?mr=243586}{MR243586}.

\bibitem{esary.proschan.ea:67:association}
Esary, J. D., Proschan, F. and Walkup, D. W. (1967).
Association of random variables, with applications.
\textit{Ann. Math. Statist.} \textbf{38}, 1466--1474.
\href{https://mathscinet.ams.org/mathscinet-getitem?mr=217826}{MR217826};
\href{https://doi.org/10.1214/aoms/1177698701}{doi:10.1214/aoms/1177698701}.

\bibitem{gu.pu:25:spatial}
Gu, Y. and Pu, F. (2025).
Spatial decorrelation of KPZ from narrow wedge.
Preprint \href{https://arxiv.org/abs/2506.23065}{arXiv:2506.23065}.

\bibitem{johansson.rahman:21:multitime}
Johansson, K. and Rahman, M. (2021).
Multitime distribution in discrete polynuclear growth.
\textit{Comm. Pure Appl. Math.} \textbf{74}, no. 12, 2561--2627.
\href{https://mathscinet.ams.org/mathscinet-getitem?mr=4373163}{MR4373163};
\href{https://doi.org/10.1002/cpa.21980}{doi:10.1002/cpa.21980};
\href{https://arxiv.org/abs/1906.01053}{arXiv:1906.01053}.

\bibitem{liu:22:multipoint}
Liu, Z. (2022).
Multipoint distribution of TASEP.
\textit{Ann. Probab.} \textbf{50}, no. 4, 1255--1321.
\href{https://doi.org/10.1214/21-AOP1557}{doi:10.1214/21-AOP1557};
\href{https://arxiv.org/abs/1907.09876}{arXiv:1907.09876}.

\bibitem{matetski.quastel.ea:21:kpz}
Matetski, K., Quastel, J. and Remenik, D. (2021).
The KPZ fixed point.
\textit{Acta Math.} \textbf{227}, no. 1, 115--203.
\href{https://mathscinet.ams.org/mathscinet-getitem?mr=4346267}{MR4346267};
\href{https://doi.org/10.4310/acta.2021.v227.n1.a3}{doi:10.4310/acta.2021.v227.n1.a3}.

\bibitem{mueller.pu:25:spatial}
Mueller, C. and Pu, F. (2025).
Spatial fluctuation for stochastic heat equation with H\"older coefficients.
Preprint \href{https://arxiv.org/abs/2510.00807}{arXiv:2510.00807}.

\bibitem{newman:80:normal}
Newman, C. M. (1980).
Normal fluctuations and the FKG inequalities.
\textit{Comm. Math. Phys.} \textbf{74}, no. 2, 119--128.
\href{https://mathscinet.ams.org/mathscinet-getitem?mr=576267}{MR576267};
\href{https://doi.org/10.1007/BF01197754}{doi:10.1007/BF01197754}.

\bibitem{nica.quastel.ea:20:one-sided*1}
Nica, M., Quastel, J. and Remenik, D. (2020).
One-sided reflected Brownian motions and the KPZ fixed point.
\textit{Forum Math. Sigma} \textbf{8}, Paper No. e63, 16 pp.
\href{https://mathscinet.ams.org/mathscinet-getitem?mr=4190063}{MR4190063};
\href{https://doi.org/10.1017/fms.2020.56}{doi:10.1017/fms.2020.56}.

\bibitem{pu:23:ergodicity}
Pu, F. (2025).
Ergodicity, CLT and asymptotic maximum of the Airy$_1$ process.
\textit{Bernoulli} \textbf{31}, 2624--2648.
Preprint \href{https://arxiv.org/abs/2311.11217}{arXiv:2311.11217}.

\bibitem{virag:20:heat}
Vir\'ag, B. (2020).
The heat and the landscape I.
Preprint \href{https://arxiv.org/abs/2008.07241}{arXiv:2008.07241}.

\bibitem{williams:91:probability}
Williams, D. (1991).
\textit{Probability with Martingales}.
Cambridge Mathematical Textbooks. Cambridge University Press, Cambridge, xvi+251 pp.
\href{https://mathscinet.ams.org/mathscinet-getitem?mr=1155402}{MR1155402};
\href{https://doi.org/10.1017/CBO9780511813658}{doi:10.1017/CBO9780511813658}.

\bibitem{widom:04:on}
Widom, H. (2004).
On asymptotics for the Airy process.
\textit{J. Statist. Phys.} \textbf{115}, no. 3--4, 1129--1134.
\href{https://mathscinet.ams.org/mathscinet-getitem?mr=2040024}{MR2040024};
\href{https://doi.org/10.1023/B:JOSS.0000022384.58696.61}{doi:10.1023/B:JOSS.0000022384.58696.61};
\href{https://arxiv.org/abs/math/0308157}{arXiv:math/0308157}.

\bibitem{wu:23:kpz}
Wu, X. (2023).
The KPZ equation and the directed landscape.
Preprint \href{https://arxiv.org/abs/2301.00547}{arXiv:2301.00547}.

\end{thebibliography}
\end{document}